\documentclass[11pt]{article}
\usepackage{amsmath}
\usepackage{amssymb}
\usepackage{amsthm}
\newtheorem{theorem}{Theorem}[section]
\newtheorem{lemma}[theorem]{Lemma}
\newtheorem{remark}[theorem]{Remark}
\newtheorem{corollary}[theorem]{Corollary}
\newtheorem{proposition}[theorem]{Proposition}
\newtheorem{problem}[theorem]{Problem}
\newtheorem{example}[theorem]{Example}

\numberwithin{equation}{section}

\long\def\symbolfootnote[#1]#2{\begingroup%
\def\thefootnote{\fnsymbol{footnote}}\footnote[#1]{#2}\endgroup}

\begin{document}

\def\C{{\mathbb C}}
\def\T{{\mathbb T}}
\def\N{{\mathbb N}}
\def\Z{{\mathbb Z}}
\def\R{{\mathbb R}}
\def\K{{\mathbb K}}
\def\CC{{\cal C}}
\def\H{{\cal H}}
\def\F{{\cal F}}
\def\X{{\cal X}}
\def\Y{{\cal Y}}
\def\epsilon{\varepsilon}
\def\kappa{\varkappa}
\def\phi{\varphi}
\def\leq{\leqslant}
\def\geq{\geqslant}
\def\re{\text{\tt Re}\,}
\def\ilim{\mathop{\hbox{$\underline{\hbox{\rm lim}}$}}\limits}
\def\dim{\hbox{\tt dim}\,}
\def\ker{\hbox{\tt ker}\,}
\def\supp{\hbox{\tt supp}\,}
\def\Re{\hbox{\tt Re}\,}
\def\ssub#1#2{#1_{{}_{{\scriptstyle #2}}}}

\title{Operators commuting with the Volterra operator are not weakly supercyclic}

\author{Stanislav Shkarin}

\date{}

\maketitle

\begin{abstract} \noindent We prove that any bounded linear operator
on $L_p[0,1]$ for $1\leq p<\infty$, commuting with the Volterra
operator $V$, is not weakly supercyclic, which answers affirmatively
a question raised by L\'eon-Saavedra and Piqueras-Lerena. It is
achieved by providing an algebraic flavored condition on an operator
which prevents it from being weakly supercyclic and is satisfied for
any operator commuting with $V$.
\end{abstract}



\section{Introduction \label{s1}}\rm

\itemsep=-2pt

All vector spaces are assumed to be over $\K$ being either the field
$\C$ of complex numbers or the field $\R$ of real numbers. As usual,
$\Z_+$ is the set of non-negative integers and $\N$ is the set of
positive integers. For a Banach space $X$, symbol $L(X)$ stands for
the space of bounded linear operators on $X$ and $X^*$ is the space
of continuous linear functionals on $X$. For $T,S\in L(X)$, we write
$[T,S]=TS-ST$ and by $\CC(T)$ we denote the {\it centralizer} of
$T$:
$$
\CC(T)=\{S\in L(X):[T,S]=0\}.
$$
We say that $Y$ is a {\it Banach space embedded into} a Banach space
$X$ if $Y$ is a linear subspace of $X$ endowed with its own norm,
which defines a topology on $Y$ stronger than the one inherited from
$X$ and turns $Y$ into a Banach space. For instance, $C[0,1]$ with
the sup-norm is a Banach space embedded into $L_1[0,1]$.

We say that $T\in L(X)$ is {\it supercyclic} (respectively, {\it
weakly supercyclic}) if there exists $x\in X$ such that the
projective orbit $O_{\rm pr}(x,T)=\{zT^nx:z\in\K,\ n\in \Z_+\}$ is
dense in $X$ (respectively, dense in $X$ with the weak topology). In
this case $x$ is said to be a {\it supercyclic vector}
(respectively, a {\it weakly supercyclic vector}) for $T$.
Supercyclicity was introduced by Hilden and Wallen \cite{hw} and
studied intensely since then. We refer to the survey \cite{msa} for
the details. Weak supercyclicity was introduced by Sanders
\cite{san} and studied in, for instance,
\cite{leon,leon2,ms1,pra,san1,58}. Gallardo and Montes \cite{gm},
answering a question raised by Salas, demonstrated that the Volterra
operator
\begin{equation}\label{volt}
Vf(x)=\int\limits_0^x f(t)\,dt,
\end{equation}
acting on $L_p[0,1]$ for $1\leq p<\infty$, is non-supercyclic. In
\cite{ms1,leon} it is shown that $V$ is not weakly supercyclic. In
\cite{55} it is proved that for $1<p<\infty$ and any non-zero $f\in
L_p[0,1]$, the sequence $V^nf/\|V^nf\|$ is weakly convergent to $0$
in $L_p[0,1]$. In \cite{mbs} and \cite{leon} it is demonstrated that
certain operators on $L_p[0,1]$, commuting with $V$, are not weakly
supercyclic. L\'eon and Piqueras \cite{leon} raised a question
whether any bounded operator on $L_p[0,1]$ for $1\leq p<\infty$,
commuting with $V$, is not weakly supercyclic. In the present
article we close the issue by answering this question affirmatively.

\begin{theorem}\label{main} Let $1\leq p<\infty$ and $T\in L(L_p[0,1])$
be such that $TV=VT$. Then $T$ is not weakly supercyclic.
\end{theorem}

Our approach has nothing in common with the ones from \cite{mbs} or
\cite{leon}. In \cite{leon} it is shown that positive operators
commuting with $V$ are not weakly supercyclic and, naturally, the
proof employs positivity argument. In \cite{mbs} it is shown that
convolution operators on $L_p[0,1]$ are not weakly supercyclic
provided the convolution kernel has a nice asymptotic behavior at
zero and the proof is based upon upper and lower estimates of the
orbits. We, on the other hand, find an 'algebraic' condition on an
operator preventing it from being weakly supercyclic and demonstrate
that operators, commuting with $V$, satisfy this condition. In order
to formulate it in full generality we need to introduce the
following class of Banach spaces.

Let $\Y$ be the class of Banach spaces $X$ such that for any
sequence $\{x_n\}_{n\in\Z_+}$ in $X$ satisfying $n=O(\|x_n\|)$ as
$n\to\infty$, the set $\{x_n:n\in\Z_+\}$ is closed in the weak
topology (=weakly closed).

\begin{remark}\label{rem}\rm It immediately follows that if
$X\in\Y$ and $A\subseteq\Z_+$ is infinite, then any subset of $X$ of
the shape $\{x_n:n\in A\}$ is weakly closed provided $n=O(\|x_n\|)$
as $n\to\infty$, $n\in A$. It is also clear that a (closed linear)
subspace of a Banach space from $\Y$ also belongs to $\Y$.
\end{remark}

For sake of better understanding of the concept, we will show in the
Appendix that the class $\Y$ contains all Banach spaces of
non-trivial type. However in order to prove Theorem~\ref{main} we
just need to know that Hilbert spaces belong to $\Y$. The latter
follows from the next lemma, which appears as Proposition~5.2 in
\cite{58}.

\begin{lemma}\label{111} Let $\{x_n\}_{n\in\Z_+}$ be a sequence in a
Hilbert space $\H$ such that
$\sum\limits_{n=1}^\infty\|x_n\|^{-2}<\infty$. Then
$\{x_n:n\in\Z_+\}$ is weakly closed in $\H$. In particular, any
Hilbert space belongs to $\Y$.
\end{lemma}

It is worth noting that the above lemma is proved in \cite{58} by
applying Ball's solution \cite{ball} of the complex plank problem.

\begin{theorem}\label{cor1} Let $X$ be a Banach space and $T\in L(X)$
be such that

\begin{itemize}
\item[\rm(\ref{cor1}.1)]there exists $M\in L(X)$ satisfying
$[T,[T,M]]=0$ and such that for any cyclic vector $u$ for $T$ and
any $v\in X$, there are $B,C\in\CC(T)$ for which $C[T,M]\neq 0$ and
$Cv=Bu.$
\end{itemize}
Then $T$ is not supercyclic. If additionally
\begin{itemize}
\item[\rm(\ref{cor1}.2)] there is $R\in \CC(T)$ with dense range,
taking values in a Banach space $Y\in \Y$, embedded into $X$,
\end{itemize}
then $T$ is not weakly supercyclic.
\end{theorem}

In the case $X\in \Y$, condition (\ref{cor1}.2) is automatically
satisfied with $R=I$. Thus, we have the following corollary.

\begin{corollary}\label{cor3} Let $X\in\Y$ and $T\in
L(X)$. If $(\ref{cor1}.1)$ is satisfied, then $T$ is not weakly
supercyclic.
\end{corollary}

We would like to formulate a useful application of
Theorem~\ref{cor1}, dealing with operators on commutative Banach
algebras. Note that we do not require a Banach algebra to be unital.
For an element $a$ of a commutative Banach algebra $X$, symbol $M_a$
stands for the {\it multiplication} operator $M_a\in L(X)$, $M_a
b=ab$. We say that a commutative Banach algebra $X$ is {\it
non-degenerate} if $M_a=0$ implies $a=0$. Equivalently, $X$ is
non-degenerate if the intersection of kernels of all multiplication
operators is $\{0\}$.

\begin{theorem}\label{cor4} Let $X$ be a non-degenerate commutative Banach
algebra, $\Lambda$ be a commutative subalgebra of $L(X)$ containing
$\{M_a:a\in X\}$ and $M\in L(X)$ be such that $[M,S]\in\Lambda$ for
each $S\in\Lambda$. Then any $T\in\Lambda$ satisfying $[T,M]\neq 0$
is not supercyclic. If additionally, there is $R\in\Lambda$ with
dense range, taking values in a Banach space $Y\in \Y$, embedded
into $X$, then any $T\in\Lambda$ satisfying $[T,M]\neq 0$ is not
weakly supercyclic.
\end{theorem}

\section{Proof of main results \label{s2}}

The following lemma is a generalization of Lemma~5.5 from \cite{58}.

\begin{lemma}\label{le2}Let $\{x_n\}_{n\in\Z_+}$ be a sequence in a
topological vector space $X$, $y\in X$ and $f$ be a continuous
linear functional on $X$ such that $f(y)=1$. Assume also that $y$
belongs to the closure of the set $\Omega=\{zx_n:z\in\K,\
n\in\Z_+\}$. Then $y$ belongs to the closure of
$N=\bigl\{\frac{x_n}{f(x_n)}:n\in\Z_+,\ f(x_n)\neq 0\bigr\}$.
\end{lemma}

\begin{proof} Since $y\notin \ker f$ and $\ker f$ is closed in $X$,
$y$ belongs to the closure of $\Omega\setminus\ker f$. Since the map
$F:X\setminus\ker f\to X$, $F(u)=u/f(u)$ is continuous and $y$ is in
the closure of $\Omega\setminus\ker f$, we see that $F(y)=y$ is in
the closure of $F(\Omega\setminus\ker f)=N$, as required.
\end{proof}

The next lemma is an immediate corollary of the Angle Criterion of
supercyclicity \cite{eva}. For sake of completeness we show that it
also follows from Lemma~\ref{le2}.

\begin{lemma}\label{angle} Let $X$ be a Banach space, $\dim X>1$,
$x\in X$ and $T\in L(X)$. Assume also that there is a non-zero $f\in
X^*$ such that $f(T^nx)=o(\|T^nx\|)$ as $n\to\infty$. Then $x$ is
not a supercyclic vector for $T$.
\end{lemma}

\begin{proof} Since $\dim X>1$, we can pick $y\in X$ such that $y\notin O_{\rm
pr}(x,T)$ and $f(y)=1$. Assume that $x$ is a supercyclic vector for
$T$. Then $y$ is in the closure of $O_{\rm pr}(x,T)$. By
Lemma~\ref{le2}, $y$ is in the closure of
$N=\bigl\{u_n=T^nx/f(T^nx):n\in\Z_+,\ f(T^nx)\neq 0\bigr\}$. Since
$f(T^nx)=o(\|T^nx\|)$ as $n\to\infty$, we have $\|u_n\|\to \infty$
as $n\to\infty$. Hence $N$ is closed and therefore $y\in N\subset
O_{\rm pr}(x,T)$. We have arrived to a contradiction.
\end{proof}

\begin{lemma}\label{le3}
Let $\{x_n\}_{n\in\Z_+}$ be a sequence of elements of a Banach space
$X\in \Y$ of dimension $>1$ such that there is a non-zero $u\in X^*$
satisfying
\begin{equation}\label{le03}
|u(x_n)|=O(n^{-1}\|x_n\|)\ \ \text{as $n\to\infty$}.
\end{equation}
Then the weak closure $\widetilde \Omega$ of the set
$\Omega=\{zx_n:z\in\K,\ n\in\Z_+\}$ is norm nowhere dense in $X$. In
particular, $\Omega$ is not weakly dense in $X$.
\end{lemma}

\begin{proof} Assume the contrary. Then $\widetilde \Omega$
contains a non-empty norm open set $W$. Since $u\neq 0$, $\dim X\geq
2$ and $\Omega$ is a countable union of one-dimensional subspaces of
$X$, we can pick $y\in W\setminus\Omega$ such that $b=u(y)\neq 0$.
Let $f=b^{-1}u$. Then $f\in X^*$ and $f(y)=1$. By Lemma~\ref{le2},
$y$ belongs to the weak closure of the set $N=\{y_n:n\in\Z_+,\
f(x_n)\neq 0\}$, where $y_n=\frac{x_n}{f(x_n)}$. According to
(\ref{le03}), $n=O(\|y_n\|)$ as $n\to\infty$. Since $X\in\Y$,
Remark~\ref{rem} implies that $N$ is weakly closed in $\H$. Hence
$y\in N\subseteq \Omega$. We have arrived to a contradiction.
\end{proof}

The following lemma is the key tool in the proof of
Theorem~\ref{cor1}.

\begin{lemma}\label{general} Let $X$ be a Banach space, $x\in X$ and $T\in L(X)$.
Assume also that $M\in L(X)$ and $C,B,R\in \CC(T)$ are such that
$S=[T,M]\in \CC(T)$ and $CMRx=BRx$. Then
\begin{equation}\label{g1}
\text{$|S^*C^*h(RT^nx)|\leq \|(B-CM)^*h\|(n+1)^{-1}\|RT^nx\|$ for
any $n\in\Z_+$ and $h\in X^*$.}
\end{equation}
\end{lemma}

\begin{proof} Since $TM-MT=S$ commutes with $T$, an elementary
inductive argument shows that
\begin{equation}\label{der}
\text{$T^nM-MT^n=nST^{n-1}$ for each $n\in\N$.}
\end{equation}

Let $h\in X^*$, $g=S^*C^*h$ and $y=Rx$. Then for each $n\in\Z_+$, we
have
$$
g(T^ny)=S^*C^*h(T^ny)=h(CST^ny).
$$
Using (\ref{der}) and the equality $CT=TC$, we obtain
$$
(n+1)g(T^ny)=(n+1)h(CST^ny)=h(CT^nMy-CMT^ny)=h(T^nCMy-CMT^ny).
$$
Since $CMy=By$ and $BT=TB$, we get
$$
(n+1)g(T^ny)=h(T^nBy-CMT^ny)=h(BT^ny-CMT^ny)=(B-CM)^*h(T^ny).
$$
The above display immediately implies that
$$
(n+1)|S^*C^*h(T^ny)|=(n+1)|g(T^ny)|\leq \|(B-CM)^*h\|\|T^ny\|\ \
\text{for any $n\in\Z_+$}.
$$
Since $y=Rx$ and $TR=RT$, we see that $T^ny=T^nRx=RT^nx$ and
(\ref{g1}) follows. \end{proof}

\subsection{Proof of Theorem~\ref{cor1}}

First, assume that only condition (\ref{cor1}.1) is satisfied. Let
$M\in L(X)$ be the operator provided by (\ref{cor1}.1) and
$S=[T,M]$. Then $S\in\CC(T)$. Assume that $T$ has a supercyclic
vector $u\in X$. By (\ref{cor1}.1) with $v=Mu$ we can pick
$B,C\in\CC(T)$ such that $CS\neq 0$ and $CMu=Bu$. All conditions of
Lemma~\ref{general} with $R=I$ and $x=u$ are satisfied. Since
$CS\neq 0$, we have $(CS)^*=S^*C^*\neq 0$ and we can pick $h\in X^*$
such that $g=S^*C^*h\neq 0$. According to Lemma~\ref{general},
$g(T^nu)=O(n^{-1}\|T^nu\|)=o(\|T^nu\|)$ as $n\to\infty$. By
Lemma~\ref{angle}, $u$ can not be a supercyclic vector for $T$. We
have arrived to a contradiction.

Assume now that (\ref{cor1}.1) and (\ref{cor1}.2) are satisfied. Let
$M$ and $R$ be operators provided by (\ref{cor1}.1) and
(\ref{cor1}.2). Then $S=[T,M]\in\CC(T)$. Since a subspace of an
element of $\Y$ belongs to $\Y$, we, replacing $Y$ by the closure of
$R(X)$ in $Y$, if necessary, can assume that $R(X)$ is dense in $Y$.
Assume that $T$ has a weakly supercyclic vector $x\in X$. Then $x$
is cyclic for $T$ and since $R$ has dense range and commutes with
$T$, $u=Rx$ is also a cyclic vector for $T$. By (\ref{cor1}.1), we
can pick $B,C\in\CC(T)$ such that $CS\neq 0$ and $CMu=Bu$. That is,
$CMRx=BRx$. Thus all conditions of Lemma~\ref{general} are
satisfied. Since $CS\neq 0$, we have $(CS)^*=S^*C^*\neq 0$ and we
can pick $h\in X^*$ such that $g=S^*C^*h\neq 0$. From the closed
graph theorem it follows that $R:X\to Y$ is continuous. Since $x$ is
a weakly supercyclic vector for $T$, $O_{\rm pr}(x,T)$ is weakly
dense in $X$. Since $R:X\to Y$ is continuous and has dense range,
$\Omega=R(O_{\rm pr}(x,T))$ is weakly dense in $Y$. According to
Lemma~\ref{general},
$g(RT^nx)=O(n^{-1}\|RT^nx\|)=O(n^{-1}\|RT^nx\|_{Y})$ as
$n\to\infty$. Then Lemma~\ref{le3} implies that the set
$\{zRT^nx:z\in\K,\ n\in\Z_+\}=\Omega$ is not weakly dense in $Y$.
This contradiction completes the proof.

\subsection{Proof of Theorem~\ref{cor4}}

Let $T\in \Lambda$ and $S=[T,M]\neq 0$. Since $\Lambda$ is closed
under the operator $A\mapsto[A,M]$, we have $S\in\Lambda$. Since
$\Lambda$ is commutative, $[T,S]=[T,[T,M]]=0$. Let now $u\in X$ be a
cyclic vector for $T$, $v\in X$ and $C=M_u$, $B=M_v$. Since
$M_a\in\Lambda$ for any $a\in X$ and $\Lambda$ is commutative, we
have $B,C\in \CC(T)$. Moreover, since $X$ is commutative, we have
$Cv=uv=vu=Bu$. Next, we shall show that $C$ is injective. Let $y\in
X$ be such that $Cy=0$. That is, $uy=yu=0$. Since $T$ commutes with
$M_y$, we get $yT^nu=T^n(yu)=0$ for each $n\in\Z_+$. Since $u$ is
cyclic for $T$, we then have $yx=0$ for each $x\in X$. Since $X$ is
non-degenerate, we have $y=0$. Hence $C$ is injective. Since $S\neq
0$, $CS=C[T,M]\neq 0$. Thus (\ref{cor1}.1) is satisfied. According
to Theorem~\ref{cor1}, $T$ is non-supercyclic. Assume additionally
that there is $R\in\Lambda$ with dense range, taking values in a
Banach space $Y\in \Y$, embedded into $X$. Since $\Lambda$ is
commutative, $R\in\CC(T)$ and (\ref{cor1}.2) is satisfied. By
Theorem~\ref{cor1}, $T$ is not weakly supercyclic. The proof is
complete.

\subsection{Proof of Theorem~\ref{main}}

For $1\leq p\leq \infty$, consider the multiplication by the
argument operator $M$ on $L_p[0,1]$:
\begin{equation}\label{mult}
Mf(x)=xf(x).
\end{equation}

\begin{lemma}\label{mv1} Let $1\leq p<\infty$ and $V,M$ be the
operators on $L_p[0,1]$ defined in $(\ref{volt})$ and
$(\ref{mult})$. Then $\CC(V)\cap \CC(M)=\{cI:c\in\K\}$.
\end{lemma}

\begin{proof}Clearly $\{cI:c\in\K\}\subseteq \CC(V)\cap \CC(M)$.
Assume that $T\in \CC(V)\cap \CC(M)$. Then $TM^n=M^nT$ for each
$n\in\Z_+$. Applying this operator equalities to the function $\bf
1$, being identically $1$, we see that $Tu_n=u_nf$, where $f=T{\bf
1}$ and $u_n(x)=x^n$. Hence $Tp=pf$ for any polynomial $p$. Since
the space of polynomials is dense in $L_p[0,1]$, we see that $f\in
L_\infty[0,1]$ and $T$ is the operator of multiplication by $f$.
Since $T$ commutes with $V$, we have $Mf=Tu_1=TV{\bf 1}=VT{\bf
1}=Vf$. Since only constant functions $g$ satisfy the equality
$Mg=Vg$, $f$ is constant. Hence $T=cI$ for some $c\in\K$. This
proves the required equality.
\end{proof}

\begin{lemma}\label{mvo} Let $1\leq p<\infty$, $T\in
L(L_p[0,1])$, $TV=VT$ and $S=TM-MT$. Then $SV=VS$.
\end{lemma}

\begin{proof} Clearly
$$
SV-VS=TMV-MTV-VTM+VMT=T(MV-VM)-(MV-VM)T.
$$
It is straightforward to verify that $MV-VM=V^2$. Thus,
$SV-VS=TV^2-V^2T=0$.
\end{proof}

The Volterra algebra is the Banach space $L_1[0,1]$ with the
convolution multiplication
$$
f\star g(x)=\int\limits_0^x f(t)g(x-t)\,dt.
$$
According to well-known properties of convolution, the above
integral converges for almost all $x\in[0,1]$, $f\star g\in
L_1[0,1]$ and $\|f\star g\|_1\leq \|f\|_1\|g\|_1$. It also follows
that $\star$ is associative and commutative: $f\star g=g\star f$ and
$(f\star g)\star h=f\star(g\star h)$ for any $f,g,h\in L_1[0,1]$.
Moreover, for each $p\in [1,\infty]$, $\|f\star g\|_p\leq
\|f\|_1\|g\|_p$ for any $f\in L_1[0,1]$ and $g\in L_p[0,1]$. In
particular $\star$ turns each $L_p[0,1]$ into a commutative Banach
algebra. It is also easy to see that the Volterra operator acts
according to the formula $Vf=f\star {\bf 1}$. Thus, $V$ is an
injective multiplication operator and therefore each
$(L_p[0,1],\star)$ is non-degenerate. We will also need the
following result of J.~Erd\"os \cite{erd1,erd2}, see \cite{ms2} for
a different proof.

\begin{lemma}\label{comm} Let $1\leq p<\infty$. Then the subalgebra
$\CC(V)$ of $L(L_p[0,1])$ is commutative.
\end{lemma}

We are ready to prove Theorem~\ref{main}. Let $1\leq p<\infty$. Then
the Banach algebra $X_p$, being the Banach space $L_p[0,1]$ with the
multiplication $\star$ is commutative. Since $V$ is injective and is
a multiplication (by ${\bf 1}$) operator on $X_p$, the algebra $X_p$
is non-degenerate. By Lemma~\ref{comm}, $\CC(V)$ is a commutative
subalgebra of $L(X_p)$. Since $X_p$ is commutative and $V$ is a
multiplication operator on $X_p$, $\CC(V)$ contains all the
multiplication operators. By Lemma~\ref{mvo}, $QM-MQ\in\CC(V)$ for
any $Q\in\CC(V)$, where $M$ is the operator defined by (\ref{mult}).
Clearly, $R=V^2\in \CC(V)$ has dense range and takes values in the
Hilbert space $W^{1,2}_0[0,1]$ being the Sobolev space of absolutely
continuous functions on $[0,1]$ vanishing at $0$ with square
integrable first derivative. Since $W^{1,2}_0[0,1]$ is embedded into
$L_p[0,1]$ and any Hilbert space belongs $\Y$ according to
Lemma~\ref{111}, Theorem~\ref{cor4} implies that any $T\in\CC(V)$
such that $[T,M]\neq 0$ is not weakly supercyclic. Let now
$T\in\CC(V)$ and $[T,M]=0$. By Lemma~\ref{mv1}, $T$ is a scalar
multiple of the identity and therefore is non-cyclic. Thus any
$T\in\CC(V)$ is not weakly supercyclic. The proof of
Theorem~\ref{main} is complete.

\section{Multiplication operators on Banach algebras}

From Theorem~\ref{main} it follows that multiplication operators on
the Volterra algebra are not weakly supercyclic. Motivated by this
observation, we would like to raise the following question.

\begin{problem}\label{q1} Can a multiplication operator on a commutative Banach
algebra be supercyclic or at least weakly supercyclic? Does anything
change if we add a scalar multiple of identity to a multiplication
operator?
\end{problem}

The rest of the section is devoted to the discussion of this
problem. Recall that a continuous linear operator $T$ on a
topological vector space $X$ is called {\it hypercyclic} if there is
$x\in X$ for which the orbit $O(x,T)=\{T^nx:n\in\Z_+\}$ is dense in
$X$. If $X$ is a Banach space, then $T$ is called {\it weakly
hypercyclic} if it is hypercyclic as an operator on $X$ with weak
topology. The next example shows that in the non-commutative setting
the answer to the above question is affirmative.

\begin{example}\label{HS} Let $\H$ be the Hilbert space of
Hilbert--Schmidt operators on $\ell_2$. With respect to the
composition multiplication, $\H$ is a Banach algebra. Let also
$S\in\H$ be defined by its action on the basic vectors as follows:
$Se_0=0$, $Se_n=n^{-1}e_{n-1}$ if $n\geq 1$. Consider the left
multiplication by $S$ operator $\Phi\in L(\H)$, $\Phi(T)=ST$. Then
$\Phi$ is supercyclic and $I+\Phi$ is hypercyclic.
\end{example}

\begin{proof} It is easy to see that $X_n=\ker \Phi^n\cap
\Phi^n(\H)$ is the set of $T\in\H$ such that $T(\ell_2)$ is
contained in the linear span of $\{e_0,\dots,e_{n-1}\}$. It is easy
to see that the union of $X_n$ is dense in $\H$. In \cite{123} it is
shown that the last property implies hypercyclicity of $I+\Phi$ and
supercyclicity of $\Phi$.
\end{proof}

It is worth noting that supercyclicity of the above operator $\Phi$
follows also from the Supercyclicity Criterion \cite{msa}.
Hypercyclicity of $I+\Phi$ can also be easily derived from the main
result of \cite{bmp}. The following two propositions highlight the
difficulties of Problem~\ref{q1}.

\begin{proposition}\label{hypab} Let $X$ be a commutative complex Banach
algebra. Then for any $a\in X$, the operator $M_a$ is not weakly
hypercyclic.
\end{proposition}

\begin{proof} Let $a\in X$. First, consider the case when there
exists a non-zero character $\kappa:X\to\C$. It is easy to see that
$\kappa$ is an eigenvector of $M_a^*$. Since the point spectrum of
the dual of any weakly hypercyclic operator is empty, $M_a$ is not
weakly hypercyclic. It remains to consider the case when there are
no non-zero characters on $X$. Since a commutative complex Banach
algebra with no non-zero characters is radical \cite{hel}, $X$ is a
radical algebra. Hence the spectrum of $a$ is $\{0\}$ and therefore
$M_a$ is quasinilpotent. It follows that each orbit of $M_a$ is a
sequence norm-convergent to $0$. Thus, $M_a$ can not be weakly
hypercyclic.
\end{proof}

\begin{proposition}\label{supab} Let $X$ be a commutative complex
Banach algebra of dimension $>1$. Assume also that there exists
$a\in X$ such that $M_a$ is weakly supercyclic. Then $X$ is radical.
\end{proposition}

\begin{proof}Assume that $X$ is non-radical. Then there exists a
non-zero character $\kappa:X\to\C$. Hence $H=\ker\kappa$ is an
$M_a$-invariant closed hyperplane in $X$. As well-known, if a
supercyclic operator on a topological vector space has an invariant
closed hyperplane, then the restriction of some its scalar multiple
to this hyperplane is hypercyclic, see, for instance, \cite{shsh}.
Thus, replacing $a$ by $\lambda a$ for some non-zero $\lambda\in\C$,
if necessary, we can assume that the restriction of $M_a$ to $H$ is
weakly hypercyclic. We have arrived to a contradiction with
Proposition~\ref{hypab}.
\end{proof}

Finally, let $X$ be a non-degenerate commutative Banach algebra and
$\Lambda=\{cI+M_a:c\in\K,\ a\in X\}$. Recall that a {\it derivation}
on $X$ is $M\in L(X)$ satisfying $M(ab)=(Ma)b+a(Mb)$. It is easy to
see that for any derivation $M$ on $X$ the operator $A\mapsto [M,A]$
preserves $\Lambda$. Moreover, $M$ commutes with $M_a$ if and only
if $Ma=0$. Thus, by Theorem~\ref{cor4}, $cI+M_a$ is non-supercyclic
provided there is a derivation on $X$, which does not annihilate
$a$. Combining this with the second part of Theorem~\ref{cor1}, we
have the following proposition.

\begin{proposition}\label{pro1} Let $X$ be a non-degenerate
commutative Banach algebra and $a\in X$. Assume also that there
exists a derivation $M$ on $X$ for which $Ma\neq 0$. Then $cI+M_a$
is non-supercyclic for any $c\in\K$. If additionally there exists
$R\in L(X)$ with dense range commuting with $M_a$ and taking values
in a Banach space $Y\in\Y$ embedded into $X$, then $cI+M_a$ is not
weakly supercyclic.
\end{proposition}

From Propositions~\ref{pro1} and~\ref{supab} it follows that the
only place to look for an affirmative answer to Problem~\ref{q1} are
radical commutative Banach algebras with few derivations.

\section{Concluding Remarks}

\noindent {\bf 1.} \ The following example shows that the second
part of condition (\ref{cor1}.1) is essential.

\begin{example} \label{exa} Let $\{e_n\}_{n\in\Z_+}$ be the
canonical orthonormal basis in $\ell_2$ and $T,M\in L(\ell_2)$ be
defined by: $Te_0=e_0$, $Te_n=e_n+n^{-1}e_{n-1}$ for $n\geq 1$,
$Me_0=0$ and $Me_n=e_{n-1}$ for $n\geq 1$. Then $T$ is hypercyclic,
$[T,M]$ has dense range and $[T,[T,M]]=0$.
\end{example}

\begin{proof} Since $T$ is the sum of the identity operator and a
backward weighted shift, $T$ is hypercyclic according to Salas
\cite{salas}. Computing the values of operators on basic vectors,
one can easily see that $[T,M]=(T-I)^2$ and therefore $[T,[T,M]]=0$.
Finally, it is straightforward to see that the range of
$[T,M]=(T-I)^2$ contains all basic vectors. Hence $[T,M]$ has dense
range.
\end{proof}

\medskip

\noindent {\bf 2.} \ Since any operator from $\CC(V)$ fails to be
weakly supercyclic, the following question becomes interesting.
Assume that the Volterra operator $V$ acts on $L_2[0,1]$.

\begin{problem}\label{q2} Can we find operators $A,B\in \CC(V)$ and
$f\in L_2[0,1]$ for which the set $\{zA^nB^mf:z\in\K,\ m,n\in\Z_+\}$
is dense $($or at least weakly dense$)$ in $L_2[0,1]$? In other
words, can a $2$-generated subsemigroup of $\CC(V)$ be supercyclic
or at least weakly supercyclic?
\end{problem}

\medskip

\noindent {\bf 3.} \ The following notion was introduced by Enflo
\cite{enflo}. Let $X$ be a Banach space and $n\in\N$. We say that
$T\in L(X)$ is {\it cyclic with support $n$} if there is $x\in X$
such that the set $\{a_1T^{k_1}x+{\dots}+a_nT^{k_n}x:a_j\in\K,\
k_j\in\Z_+\}$ is dense in $X$. Clearly cyclicity with support 1 is
exactly supercyclicity. In \cite{leon2} it is claimed that there are
no known examples of an operator on an infinite dimensional Banach
space, cyclic with support $n\geq 2$ and non-supercyclic. The next
example shows that such operators do exist. It is a modification of
an example constructed in \cite{bm}. By $\T$ we denote the unit
circle in $\C$: $\T=\{z\in\C:|z|=1\}$. We also consider the
functions $u_n:\T\to\T$, $u_n(z)=z^n$ for $n\in\Z$. Recall that a
non-empty compact subset $K$ of $\T$ is a {\it Kronecker set} if
$\{u_n:n\in\Z_+\}$ is dense in $C(K,\T)$ with respect to the uniform
metric $d(f,g)=\max\{|f(z)-g(z)|:z\in\T\}$. It is well-known that
there are infinite Kronecker sets and even Kronecker sets with no
isolated points. In the latter case the Banach space $C(K)$ is
infinite dimensional.

\begin{example}\label{exam} Let $K\subset\T$ be a
Kronecker set of cardinality $>1$ and $T\in L(C(K))$, $Tf(z)=zf(z)$.
Then $T$ is not weakly supercyclic, while the set $\{aT^{k}{\bf
1}+aT^{m}{\bf 1}:a\in[0,\infty),\ k,m\in\Z_+\}$ is dense in $C(K)$.
\end{example}

\begin{proof}Since $K$ is a Kronecker set, $K$ has empty interior in
$\T$ and therefore the set $W=\{f\in C(K):f(K)\ \ \text{is
finite}\}$ is dense in $C(K)$. Let $g\in W$. Using the fact that any
$z\in\C$ with $|z|\leq 2$ is a sum of two numbers from $\T$, we can
find $g_1,g_2\in W$ such that $|g_1|=|g_2|={\bf 1}$ and
$g=r(g_1+g_2)$, where $r=\|g\|/2$. Since $g_j\in C(K,\T)$ and $K$ is
the Kronecker set, there are strictly increasing sequences
$\{k_n\}_{n\in\Z_+}$ and $\{m_n\}_{n\in\Z_+}$ such that
$f_{k_n}=T^{k_n}{\bf 1}$ converges uniformly to $g_1$ and
$f_{m_n}=T^{m_n}{\bf 1}$ converges uniformly to $g_2$ as
$n\to\infty$. Hence $r(T^{k_n}{\bf 1}+T^{m_n}{\bf 1})$ converges to
$g$ in $C(K)$. Since $g$ is an arbitrary element of $W$ and $W$ is
dense in $C(K)$, we see that $\{aT^{k}{\bf 1}+aT^{m}{\bf
1}:a\in[0,\infty),\ k,m\in\Z_+\}$ is dense in $C(K)$. Now we show
that $T$ is not weakly supercyclic. Indeed, let $f\in C(K)$, $f\neq
0$. Pick $s,t\in K$ such that $s\neq t$ and $f(s)\neq 0$. Consider
the set $G=\{g\in C(K):|g(t)f(s)|>|g(s)f(t)|\}$. Clearly $G$ is
non-empty and weakly open in $C(K)$. On the other hand, it is easy
to see that for any element $g$ of $O_{\rm pr}(f,T)$,
$|g(t)f(s)|=|g(s)f(t)|$. It follows that $O_{\rm pr}(f,T)$ does not
meet $G$ and therefore can not be weakly dense in $C(K)$. Hence $T$
is not weakly supercyclic.
\end{proof}

\section{Appendices} \normalsize

\small

\subsection{Related operators that are not weakly
supercyclic}

In \cite{leon1} it is observed that for any $p\in(1,\infty)$ the
Ces\`aro operator
$$
Cf(x)=\frac1x\int_0^xf(t)\,dt,
$$
acting on $L_p[0,1]$, is hypercyclic. Clearly $C$ is the composition
of the Volterra operator and the (unbounded) operator of
multiplication by the function $\alpha(x)=x^{-1}$. The fact that
$\alpha$ is non-integrable turns out to be the reason for this
phenomenon.

\begin{proposition}\label{thma1} Let $1\leq p<\infty$, $\alpha\in L_1[0,1]$, $\alpha\geq
0$ and assume that the formula
$$
T_\alpha f(x)=\alpha(x)\int_0^x f(t)\,dt
$$
defines a bounded linear operator on $L_p[0,1]$. Then $T_\alpha$ is
not weakly supercyclic.
\end{proposition}

We prove the above proposition by applying the {\it Comparison
Principle} for weak supercyclicity. Namely, assume that $X$ and $Y$
are Banach spaces, $T\in L(X)$ and $S\in L(Y)$ are such that there
exists a bounded linear operator $J:X\to Y$ with dense range
satisfying $JT=SJ$. Then weak supercyclicity of $T$ implies weak
supercyclicity of $S$. It follows from the equality $O_{\rm
pr}(Jx,S)=J(O_{\rm pr}(x,T))$, which implies that $Jx$ is a weakly
supercyclic vector for $S$ provided $x$ is a weakly supercyclic
vector for $T$.

\begin{proof}[Proof of Proposition~$\ref{thma1}$] If $\alpha$ vanishes
on a set of positive measure, then the range of $T_\alpha$ is
non-dense and therefore $T_\alpha$ is not weakly supercyclic
\cite{san}. Thus, we can assume that almost everywhere
$\alpha(x)>0$. Hence the continuous function $h=V\alpha$ is strictly
increasing. Multiplying $\alpha$ by a positive constant, if
necessary, we can without loss of generality assume that $h(1)=1$.
Since $h(0)=0$, $h$ provides an increasing autohomeomorphism of the
interval $[0,1]$. Let $\phi:[0,1]\to [0,1]$ be the inverse of $h$:
$\phi=h^{-1}$.

For any $f\in L_p[0,1]$, we, using the change of variables
$t=\phi^{-1}(s)$, obtain
$$
\int_0^1 \frac{|f(\phi(t))|}{\alpha(\phi(t))}\,dt= \int_0^1
\frac{|f(s)|}{\alpha(s)}(\phi^{-1})'(s)\,ds=\int_0^1
|f(s)|\,ds=\|f\|_1\leq \|f\|_p.
$$
Hence the formula $Jf(x)=\frac{f(\phi(x))}{\alpha(\phi(x))}$ defines
a bounded linear operator from $L_p[0,1]$ to $L_1[0,1]$. Using the
same change of variables, it is straightforward to see that
$JT_\alpha=VJ$, where $V$ is the Volterra operator acting on
$L_1[0,1]$. Since $J$ has dense range and $V$ is not weakly
supercyclic, we aplying Comparison Principle see that $T_\alpha$ is
not weakly supercyclic.
\end{proof}

In particular, from Proposition~\ref{thma1} it follows that for any
$s>-1$ the operator
$$
R_sf(x)=x^s\int_0^xf(t)\,dt,
$$
acting on $L_p[0,1]$ for $1\leq p<\infty$, is not weakly
supercyclic. On the other hand, $R_{-1}$ coincides with the
hypercyclic Ces\`aro operator $C$.

In a similar way one can treat weighted Volterra operators with a
positive weight. Let $1\leq p,q\leq\infty$ be such that
$\frac1p+\frac1q=1$ and $\alpha\in L_q[0,1]$. Then it is easy to see
that formula
$$
S_\alpha f(x)=\int_0^x \alpha(t)f(t)\,dt
$$
defines a bounded linear operator on $L_p[0,1]$.

\begin{proposition}\label{thma2}\it Let $1\leq p<\infty$,
$1<q\leq\infty$ be such that $\frac1p+\frac1q=1$ and $\alpha\in
L_q[0,1]$ be almost everywhere positive. Then the operator
$S_\alpha$ is not weakly supercyclic.
\end{proposition}

\begin{proof} Consider the bounded linear operator $M_\alpha:L_p[0,1]\to
L_1[0,1]$, $M_\alpha f(x)=\alpha(x)f(x)$. Since $\alpha$ is positive
almost everywhere, the operator $M_\alpha$ has dense range. Since
$\alpha\in L_1[0,1]$, the operator $T_\alpha$ defined in
Proposition~\ref{thma1} acts boundedly on $L_1[0,1]$. From the
definitions of $T_\alpha$, $M_\alpha$ and $S_\alpha$ it immediately
follows that $M_\alpha S_\alpha=T_\alpha M_\alpha$. By
Proposition~\ref{thma1}, $T_\alpha$ acting on $L_1[0,1]$ is not
weakly supercyclic. Applying Comparison Principle, we see that
$S_\alpha$ is not weakly supercyclic.
\end{proof}

\subsection{Weakly closed sequences}

For potential applications of Theorem~\ref{cor1}, it could be useful
to have more information on which spaces belong to $\Y$. For $1\leq
p\leq 2$, by $\X_p$ we denote the class of Banach spaces $X$ such
that for any sequence $\{x_n\}_{n\in\Z_+}$ satisfying
$\sum\limits_{n=0}^\infty \|x_n\|^{-p}<\infty$, the set
$\{x_n:n\in\Z_+\}$ is weakly closed. We also denote
$$
\X=\bigcup_{1<p\leq 2} \X_p.
$$
Obviously $\X\subseteq \Y$ and $X_p\supseteq \X_q$ if $p\leq q$. The
following observations are made in \cite{58}.
\begin{align}
&\text{$\X_1$ is the class of all Banach spaces and $c_0\notin
\X$;}\label{xp2}
\\
&\text{if $1<p<\infty$, $\textstyle\frac1p+\frac1q=1$ and $1\leq
r<\min\{2,q\}$, then $\ell_p\in \X_r$.}\label{xp3}
\end{align}

In this section we extend (\ref{xp3}).  The general idea remains the
same as in \cite{58}, although a number of significant details
differ. For the rest of the section let $(\Omega,\F,P)$ be a
probability space and $\gamma_n:\Omega\to\R$ for $n\in\Z_+$ be
independent standard normal random variables.

\begin{lemma}\label{wc} Let $\{x_n\}_{n\in\Z_+}$ be a sequence in a
Banach space $X$. Assume that there exists a sequence
$\{f_n\}_{n\in\Z_+}$ in $X^*$ such that
\begin{itemize}
\item[\rm(\ref{wc}.1)]$\sum\limits_{n=0}^\infty \|f_n\|^2<\infty$ and the
series $\sum\limits_{n=0}^\infty \gamma_n f_n$ converges in
$L_1(\Omega,X^*)$;
\item[\rm(\ref{wc}.2)]the sequence $\{|f_n(x_n)|^{-1}\}_{n\in\Z_+}$ belongs to
$\bigcup\limits_{1\leq p<\infty}\ell_p$.
\end{itemize}
Then $0$ is not in the weak closure of the set $\{x_n:n\in\Z_+\}$.
\end{lemma}

\begin{proof}Without loss of generality, we can assume that the space
$X$ is over $\R$. Indeed, otherwise we simply consider $X$ as an
$\R$-linear Banach space and replace the $\C$-linear functionals
$f_n$ by the $\R$-linear functionals $g_n={\tt Re}\,(c_nf_n)$, where
$c_n=f_n(x_n)^{-1}|f_n(x_n)|$. We can also assume that the function
$\phi:\Omega\to X^*$, $\phi=\sum\limits_{n=0}^\infty \gamma_nf_n$ is
Borel measurable. Obviously $\mu(A)=P(\phi^{-1}(A))$ is a Borel
$\sigma$-additive probability measure on $X$. For any $n\in\N$ and
$\epsilon>0$ denote $B_{n,\epsilon}=\{g\in X^*:|g(x_n)|<\epsilon\}$.
Then $\mu(B_n)=P\{\omega:|\phi(\omega)(x_n)|<\epsilon\}$. According
to (\ref{wc}.1), we can write
\begin{equation}\label{eta1}
\mu(B_{n,\epsilon})=P\{\omega\in\Omega:|\eta_{n,\epsilon}(\omega)|<\epsilon\},
\ \ \text{where}\ \ \eta_{n,\epsilon}(\omega)=\sum_{j=1}^\infty
f_j(x_n)\gamma_j(\omega)
\end{equation}
and the series defining $\eta_{n,\epsilon}$ converges in
$L_2(\Omega)$. Clearly $\eta_{n,\epsilon}$ is a normal distribution
with zero average and standard deviance
$$
a_n=\biggl(\sum_{j=0}^\infty |f_j(x_n)|^2\biggr)^{1/2}\geq
|f_n(x_n)|.
$$
Hence $\eta_{n,\epsilon}$ has the density
$\rho_{n,\epsilon}(t)=(2\pi a_n)^{-1/2}e^{-t^2/(2a_n^2)}$. By
(\ref{eta1}),
\begin{equation}\label{estim}
\mu(B_{n,\epsilon})=\frac1{\sqrt{2\pi
a_n}}\int\limits_{-\epsilon}^{\epsilon}e^{-t^2/(2a_n^2)}\,dt=
\frac1{\sqrt{2\pi}}\int\limits_{-\epsilon/a_n}^{\epsilon/a_n}e^{-s^2/2}\,ds\leq
\frac{2\epsilon}{\sqrt{2\pi}a_n}<\epsilon a_n^{-1}\leq \epsilon
|f_n(x_n)|^{-1}.
\end{equation}
According to (\ref{wc}.2) we can pick $k\in\N$ such that
$C=\sum\limits_{n=0}^\infty |f_n(x_n)|^{-k}<\infty$. Consider the
product measure $\nu=\mu\times{\dots}\times \mu$ on $(X^*)^k$.
Clearly $\nu$ is a Borel $\sigma$-additive probability measure. Let
$$
C_{n,\epsilon}=B_{n,\epsilon}^k=\Bigl\{u\in (X^*)^k:\max_{1\leq
j\leq k}|u_j(x_n)|<\epsilon\Bigr\}.
$$
By (\ref{estim}), $\nu(C_{n,\epsilon})=\mu(B_{n,\epsilon})^k\leq
\epsilon^k|f_n(x_n)|^{-k}$. Then using the definition of $C$, we
obtain
$$
\nu(C_\epsilon)\leq \sum_{j=0}^\infty \nu(C_{n,\epsilon})\leq
\epsilon^k C,\ \ \text{where}\ \ C_\epsilon=\bigcup_{n=0}^\infty
C_{n,\epsilon}.
$$
Thus, we can choose $\epsilon>0$ small enough to ensure that
$\nu(C_\epsilon)<1$. Since $\nu((X^*)^k)=1$, we can pick $u\in
(X^*)^k\setminus C_\epsilon$. From the definition of $C_\epsilon$
and the fact that $u\notin C_\epsilon$ it immediately follows that
$x_n\notin W$ for each $n\in\Z_+$, where $W=\Bigl\{x\in
X:\max\limits_{1\leq j\leq k}|u_j(x)|<\epsilon\Bigr\}$. Since $0\in
W$ and $W$ is weakly open, we see that 0 is not in the weak closure
of $\{x_n:n\in\Z_+\}$.
\end{proof}

Recall \cite{geom} that a Banach space $X$ is said to be of {\it
type} $p\in[1,2]$ if there exists $C>0$ such that for any $n\in\N$
and any $x_1,\dots,x_n\in X$,
$$
2^{-n}\sum_{\epsilon\in\{-1,1\}^n}\biggl\|\sum_{j=1}^n
\epsilon_jx_j\biggr\|^p\leq C\sum_{j=1}^n\|x_j\|^p.
$$
According to Maurey and Pisier \cite{mp}, a Banach space $X$ is of
type $p\in(1,2]$ if and only if there exists $C>0$ such that for any
$n\in\N$ and any $x_1,\dots,x_n\in X$,
\begin{equation}\label{gauss}
\int\limits_\Omega \biggl\|\sum_{j=1}^n
\gamma_j(\omega)x_j\biggr\|^p P(d\omega)\leq C\sum_{j=1}^n\|x_j\|^p.
\end{equation}
It immediately follows that for any Banach space $X$ of type
$p\in(1,2]$ and any sequence $\{y_n\}_{n\in\Z_+}$ in $X$ such that
$\sum\limits_{n=0}^\infty \|y_n\|^p<\infty$, the series
$\sum\limits_{n=0}^\infty \gamma_n y_n$ converges in
$L_p(\Omega,X)$.

\begin{lemma}\label{wc1} Let $1<p\leq 2$ and $X$ be a Banach space
such that $X^*$ is of type $p$. Then $X\in\X_q$ for any $q\in[1,p)$.
\end{lemma}

\begin{proof} Let $1\leq q<p$ and $\{u_n\}_{n\in\Z_+}$ be a sequence
in $X$ such that $\sum\limits_{n=0}^\infty\|u_n\|^{-q}<\infty$. We
have to show that the set $A=\{u_n:n\in\Z_+\}$ is weakly closed. Let
$y\in X\setminus A$ and $x_n=u_n-y$ for $n\in\Z_+$. Clearly
$\sum\limits_{n=0}^\infty\|x_n\|^{-q}<\infty$. By Hahn--Banach
theorem, for each $n\in\Z_+$, we can pick $f_n\in X^*$ such that
$\|f_n\|=\|x_n\|^{-q/p}$ and
$f_n(x_n)=\|f_n\|\|x_n\|=\|x_n\|^{(p-q)/p}$. Since
$$
\sum_{n=0}^\infty \|f_n\|^p=\sum_{n=0}^\infty \|x_n\|^{-q}<\infty
$$
and $X^*$ is of type $p$, the series $\sum\limits_{n=0}^\infty
\gamma_n f_n$ converges in $L_p(\Omega,X^*)$ and (\ref{wc}.1) is
satisfied. Finally, since $|f_n(x_n)|=\|x_n\|^{(p-q)/p}$ and
$\sum\limits_{n=0}^\infty\|x_n\|^{-q}<\infty$, we see that the
sequence $\{|f_n(x_n)|^{-1}\}_{n\in\Z_+}$ belongs to $\ell_r$ with
$r=pq/(p-q)$. Hence, (\ref{wc}.2) is satisfied. By Lemma~\ref{wc},
$0$ is not in the weak closure of $A-y=\{x_n:n\in\Z_+\}$. That is,
$y$ is not in the weak closure of $A$. Since $y$ is an arbitrary
element of $X\setminus A$, $A$ is weakly closed.
\end{proof}

As was shown by Pisier, see \cite{geom}, $X$ is of non-trivial type
(=of type $>1$) if and only if so is $X^*$. Thus, we obtain the
following corollary.

\begin{corollary} \label{type1} Let $X$ be a Banach space of non-trivial type.
Then $X\in\X\subseteq \Y$.
\end{corollary}

In particular, since spaces $L_p[0,1]$ have type $\min\{p,2\}$ for
$1\leq p<\infty$, we see that $L_p[0,1]\in \X\subseteq\Y$ for
$1<p<\infty$. Recall \cite{geom} that a Banach space $X$ is said to
be of (finite) {\it cotype} $p\in[2,\infty)$  if there exists $C>0$
such that for any $n\in\N$ and any $x_1,\dots,x_n\in X$,
$$
2^{-n}\sum_{\epsilon\in\{-1,1\}^n}\biggl\|\sum_{j=1}^n
\epsilon_jx_j\biggr\|^p\geq C\sum_{j=1}^n\|x_j\|^p.
$$
As demonstrated by Kadets \cite{kadec}, if $X$ is a Banach space,
which is not of finite cotype, then for any sequence
$\{c_n\}_{n\in\Z_+}$ of positive numbers satisfying
$\sum\limits_{n=0}^\infty c_n^{-1}=\infty$, there is a sequence
$\{x_n\}_{n\in\Z_+}$ in $X$ such that $\|x_n\|=c_n$ for each
$n\in\Z_+$ and $0$ is in the weak closure of $\{x_n:n\in\Z_+\}$. It
follows that $X\notin\Y$. That is, any space in the class $\Y$ is of
finite cotype. Thus, ${\cal T}\subseteq \X\subseteq \Y\subseteq
{\cal C}$, where $\cal T$ is the class of Banach spaces of
non-trivial type and $\cal C$ is the class of Banach spaces of
finite cotype. It remains unclear whether some of the above
inclusions are actually equalities. A natural conjecture would be
$\X=\Y={\cal C}$.

\medskip

\rm

\vskip1truecm

\scshape

\noindent Stanislav Shkarin

\noindent Queen's University Belfast

\noindent Department of Pure Mathematics

\noindent University road, BT7 1NN \ Belfast, UK

\noindent E-mail address: \qquad {\tt s.shkarin@qub.ac.uk}

\end{document}